\documentclass[a4paper,11pt]{amsart}
\usepackage{amsmath,amsthm,amssymb,epsfig,fullpage}
\usepackage[colorlinks,linkcolor={blue},
citecolor={blue},urlcolor={red},]{hyperref}
\usepackage{hyperref}

\newtheorem{theorem}{Theorem}[section]
\newtheorem{corollary}{Corollary}
\newtheorem{proposition}{Proposition}
\newtheorem{lemma}{Lemma}
\newtheorem{remark}{Remark}

\newtheorem{definition}{Definition}
\newtheorem{assumption}{Assumption}

\numberwithin{theorem}{section}
\numberwithin{lemma}{section}
\numberwithin{proposition}{section}
\numberwithin{equation}{section}
\numberwithin{remark}{section}
\numberwithin{remarks}{section}
\numberwithin{definition}{section}
\numberwithin{scheme}{section}
\numberwithin{example}{section}
\numberwithin{corollary}{section}
\usepackage{colortbl}

\newcommand\del[1]{}

\newcommand\lin{\operatorname{lin}}

\newcommand\id{\operatorname{id}}

\newcommand{\la}{\lambda}

\def\suml{\sum\limits}

\def\intl{\int\limits}

\def\lan{\left(}
\def\ran{\right)}

\def\P{\mathbb P}

\title{Rescaling nonlinear noise for 1D stochastic parabolic equations}
\author{Ben Goldys\\
School of Mathematics and Statistics, University of Sydney, Australia \\
\and Misha Neklyudov$^\sharp$\\ Departamento de matematica, UFAM, Brasil}
\thanks{$^\sharp$Corresponding author, e-mail: misha.neklyudov@gmail.com}
\date{\today}

\begin{document}



\begin{abstract}
In this paper we show regularisation effect of nonlinear gradient noise to the solution of  1D stochastic parabolic equation. We demonstrate convergence to a martingale (independent upon space variable) when we rescale noise at the extremum points of the process.
\end{abstract}
\maketitle
 
\tableofcontents
\section{Introduction}

Regularisation by noise of partial differential equations have  been an object of intense study for a number of years, see book of Flandoli \cite{Flandoli2010},
paper of Flandoli, Gubinelli and Priola \cite{FGP2010} and, more recently,  review of the literature in Gess, Souganidis \cite{GS-CPAM-2017}. For instance, It was shown in Flandoli, Gubinelli and Priola \cite{FGP2010} that the equation
\[
du+b(x)\nabla u\,dt=\partial_x u\circ d\beta_t
\]
can be well posed even if the corresponding deterministic equation is not. In the same time, the proof was based on linearity and homogenuity of the noise. Counterexample of nonlinear equation where noise does not improve regularity is given in Flandoli \cite{Flandoli2010}.
The effect of regularization by non-linear stochastic perturbations in the setting of stochastic conservation laws has been considered in  Gess, Souganidis \cite{GS-CPAM-2017,GS-SPA-2017}, Gassiat, Gess \cite{GG2016}. The purpose of this paper to complement these results showing regularisation effect of the noise in the parabolic setting. Our estimates (Theorem \ref{thm:Mainthm_1}, propositions \ref{prop:Mainthm_2}, \ref{prop:Mainthm_3}) show that nonlinear gradient noise, scaled appropriately, leads to flattening out of the system (compare with example $2$ in section \ref{sec:counterexample} for the linear case). 



The equation we consider can be described as an Ornstein-Uhlenbeck process with the noise which is ``rescaled" at the stationary points of solution. Informally, it can be described as a limit $\epsilon\to 0$ of solution to the stochastic PDE of the form 
\begin{align}\label{Mainmodel}
d\psi^{\epsilon}&=A(\psi^{\epsilon})\,dt+g\left(\dfrac{\psi_x^{\epsilon}}{\epsilon}\right)\circ dW_t^Q,\quad x\in \mathbb{S}^1, t\geq 0,\\
\psi(0) &= \psi_0\in L^2(\mathbb{S}^1).\nonumber
\end{align}
where $\mathbb S^1$ stands for the unit circle, $\lan W_t^Q\ran$ is a white in time and coloured in space noise, stochastic integral is understood in Stratonovich sense, $A$ is a dissipative operator, and $g$ is bounded function which derivative has support concentrated mainly at zero (Precise definitions are given later). Typical example of $g$ is $g(z)=\frac{|z|}{\sqrt{1+z^2}}, z\in\mathbb{R}$. In this example it means that we are ``switching off" noise at the critical points and, as $\epsilon\to 0$, the limit of the equation \eqref{Mainmodel} is formally given by
\begin{equation}\label{Mainexample}
d\psi=A(\psi)\,dt+\id_{\{\psi_x\neq 0\}}\circ dW_t^Q,\quad x\in \mathbb{S}^1, t\geq 0.
\end{equation}
	The motivation of the setup comes from micromagnetics. It is well known \cite{Berkov2007} that the theory of stochastic Landau-Lifshitz-Gilbert equation  
	\begin{equation}
	d\mathbf{u}=(\mathbf{u}\times\triangle \mathbf{u}-\alpha\mathbf{u}\times(\mathbf{u}\times \triangle \mathbf{u}))\,dt+\nu\mathbf{u}\times\circ dW(t,x),\alpha,\nu >0,t\geq 0, x\in\mathbb{S}^1,\mathbf{u}\in \mathbb{S}^2
	\end{equation}
	(where stochastic integral is understood in Stratonovich sense) does not cover physically important case of $W$ being 3D space-time white noise. One of the ways to circumvent the problem is to consider the following toy model 
	\begin{equation}
	d\mathbf{u}=\alpha (\triangle \mathbf{u}+|\nabla \mathbf{u}|^2\mathbf{u} )\,dt+\nu\mathbf{u}^{\perp}\circ d\eta(t,x),\alpha,\nu >0,t\geq 0, x\in\mathbb{S}^1,\mathbf{u}\in \mathbb{S}^1,
	\end{equation}
	where $\mathbf{u}=(u^1,u^2)$ takes values in the circle instead of sphere,  $\mathbf{u}^{\perp}=(-u^2,u^1)$, stochastic integral is understood in Stratonovich sense and $\eta$ is 1D white in time colored in space noise. Then we have that $\mathbf{u}=e^{\imath\phi}$ and It\^o formula allows to conclude that 
	\[
	d\phi=\alpha\triangle\phi\,dt+\nu d\eta.
	\]
	Note that now $\phi$ is Ornstein-Uhlenbeck process which is well defined even if $\eta$ is space-time white noise  and, in this case, $\phi$ has enough regularity to define $e^{\imath\phi}$. Furthermore, we have Gaussian invariant measure for $\phi$ which can be transferred to invariant measure on $\mathbf{u}$. Parameters $\alpha$ and $\nu$ are connected with macroscopic temperature $T$ of the system through fluctuation-dissipation relation
	\[
	\frac{2\alpha}{\nu^2}=\frac{1}{k_B T}
	\] 
	
	Now the rescaling of $\phi$ at the extremum points can be interpreted as ``cooling off" (for the function $g=\frac{|z|}{\sqrt{1+z^2}}$) the system\footnote{for different $g$ it could also be ``heating up"} at extremum points. Our result states that such ``cooling off" (or ``heating up") at the extremum points leads to flattening out of the system i.e. we deduce that $\psi^{\epsilon}$ weakly converges to a martingale $\psi$ independent upon space variable. That seems to be of interest because we change the system only locally while the result is global.

\section{Definitions}

We identify $\mathbb{S}^1$ with semi interval $[0,2\pi)$. Let $H=L^2(\mathbb{S}^1,\mathbb{R})$ with scalar product $(\cdot,\cdot)$. Then the system 
\[
e_1=\frac{1}{\sqrt{2\pi}},\quad e_{2k+1}=\frac{1}{\sqrt{\pi}}\cos{kx},\quad e_{2k}=\frac{1}{\sqrt{\pi}}\sin{kx},\quad k\geq 1\,,
\] is an orthonormal basis in $H$. Let
\[H_l:=\lin\{e_1,\ldots,e_l\}\,,l\in\mathbb{N}\]
and let $\pi_l:H\to H_l$ denote the orthogonal projection onto $H_l$.

We will assume 
that the noise $W_t^Q$ is Hilbert-Schmidt class. 
Define 
\[W_t^Q(x)=\sum\limits_{i=1}^{\infty}q_i\beta_t^ie_i(x)\,,\]
 where $\{\beta^i_t\}_{i=1}^{\infty}$ are independent identically distributed Brownian motions and
\begin{equation}\label{eqn:QCondition}
\begin{array}{l}
\sum\limits_{l=1}^{\infty}q_l^2<\infty,\sum\limits_{l=1}^{\infty}l^2q_{2l}^2<\infty \\
q_{2k}^2=q_{2k+1}^2,k\in\mathbb{N}
\end{array}
\end{equation} 
\begin{assumption}\label{ass:a_1}
$A: \mathbb{H}^2(\mathbb{S}^1)\to H$ is a linear operator such that 
\begin{equation}\label{eqn:DissipatCondition}
(-A f,f)_{H^{1,2}(\mathbb{S}^1,dx)}\geq \alpha |f|_{H^{2,2}(\mathbb{S}^1,dx)}^2+\beta
|f|_{H^{1,2}(\mathbb{S}^1,dx)}^2,\alpha >0,\beta\in\mathbb{R}.
\end{equation}
\end{assumption}
From now on and until section \ref{sec:counterexample} we will assume that
\begin{assumption}\label{ass:a_2}
\[g, g' \in L^{\infty}(\mathbb{R}).\]
\end{assumption}

Equation \eqref{Mainmodel} can be reformulated in  the It\^o sense as follows
\begin{equation}\label{Mainmodel_Itoform}
\left\{
\begin{array}{l}
d\psi^{\epsilon} = \left(A(\psi^{\epsilon})+\dfrac{M|g'|^2(\frac{\psi_x^{\epsilon}}{\epsilon})}{2\epsilon^2}\psi^{\epsilon}_{xx}\right)\,dt+g\left(\dfrac{\psi^{\epsilon}_x}{\epsilon}\right)\, dW_t^Q,\\
\\
\psi^{\epsilon}(0) = \psi_0\in H,
\end{array}
\right.
\end{equation}
where $M=\frac{1}{2\pi}\sum\limits_{l=1}^{\infty}q_l^2$ (the calculation of It\^o correction is given in appendix).
\begin{definition}
Fix $\epsilon>0$. A progressively measurable process $\psi^{\epsilon}$ defined on a certain filtered probability space $\lan\Omega,\mathcal F,\lan\mathcal F_t\ran,\P\ran$ is said to be a weak solution of equation \eqref{Mainmodel_Itoform} if $\psi^\epsilon(\cdot)\in C\lan [0,\infty),L^2\lan\mathbb{S}^1\ran\ran$, 
for every $T>0$ 
\[\int_0^T\left|\psi^\epsilon(t)\right|_{H^{1,2}\lan\mathbb{S}^1\ran}^2dt<\infty,\quad \mathbb P-a.s.\,,\]
and for every $t>0$  and any $\phi\in C^{\infty}(\mathbb{S}^1)$
\begin{align}
(\psi^{\epsilon}(t),\phi) &= (\psi_0,\phi)+\int\limits_0^t
\left(\psi^{\epsilon},A^*(\phi)\right)\,ds-\frac{M}{2}\int\limits_0^t\left(\frac{\phi_x}{\epsilon},G\lan\frac{\psi_x^{\epsilon}}{\epsilon}\ran\right)\,ds
\nonumber\\
&+\int\limits_0^t\left(g\left(\dfrac{\psi^{\epsilon}_x}{\epsilon}\right),\phi\right) dW_s^Q,\,\,\,\,\,\P-a.s.,\nonumber
\end{align}
where 
\begin{equation}\label{eqn:DefG}
G(x)=\intl_0^{x}|g'|^2(y)\,dy, x\in \mathbb{R}.
\end{equation}
We say that $\psi^\epsilon$ is a strong solution if it is a weak solution, such that $\psi^\epsilon(\cdot)\in C\lan [0,\infty),H^{1,2}\lan\mathbb{S}^1\ran\ran$, 
for every $T>0$ 
\[\int_0^T\left|\psi^\epsilon(t)\right|_{H^{2,2}\lan\mathbb{S}^1\ran}^2dt<\infty,\quad \mathbb P-a.s.\,,\]
and for every $t>0$  
\begin{align}
\psi^{\epsilon}(t) &= \psi_0 +\int\limits_0^t\left( A(\psi^{\epsilon})+\dfrac{M|g'|^2(\frac{\psi_x^{\epsilon}}{\epsilon})}{2\epsilon^2}\psi^{\epsilon}_{xx}\right)\,ds\nonumber\\
&+\int\limits_0^t g\left(\dfrac{\psi^{\epsilon}_x}{\epsilon}\right) dW_s^Q,\,\,\,\,\P-a.s.
\end{align}
\end{definition}
We will denote 
\begin{align*}
A^{\epsilon} &: \mathbb{H}^2(\mathbb{S}^1)\to H,\quad  A^{\epsilon}(f):=A(f)+\dfrac{M|g'|^2(\frac{f_x}{\epsilon})}{2\epsilon^2}f_{xx}\\
\sigma^{\epsilon} &: \mathbb{H}^1(\mathbb{S}^1)\to L^{\infty}(\mathbb{S}^1),\quad\sigma^{\epsilon}(f):=g\left(\dfrac{f_x}{\epsilon}\right).
\end{align*}
We define the Galerkin approximation of equation \eqref{Mainmodel} as follows
\begin{equation}\label{eqn:GalerkinApprox}
\left\{
\begin{array}{l}
d\psi^{m,\epsilon} =\pi_m(A^{\epsilon}(\pi_m\psi^{m,\epsilon}))\,dt+\pi_m(\sigma^{\epsilon}(\pi_m\psi^{m,\epsilon})dW_s^Q),\\
\\
\psi^{m,\epsilon}(0)=\pi_m\psi_0
\end{array}\right.
\end{equation}
Equation \eqref{eqn:GalerkinApprox} is an SDE with continuous coefficients and therefore has a local solution. 

\section{A Priori Estimates}

In the following proposition we will deduce energy estimates uniform in $\epsilon$ and $m$ to conclude existence of a global solution of equation \eqref{eqn:GalerkinApprox}.
\begin{proposition}\label{prop:Galerkin_Estimate}
For every $\epsilon>0$, $t>0$ and any $m=1,2,\ldots$
\begin{align}
\mathbb{E}|\psi^{m,\epsilon}|_H^2(t)&-2 \mathbb{E}\intl_0^t(A\psi^{m,\epsilon},\psi^{m,\epsilon})_H\,ds\nonumber\\
&+M\mathbb{E}\intl_0^t\intl_{\mathbb{S}^1}\frac{\psi_x^{m,\epsilon}}{\epsilon} G(\frac{\psi_x^{m,\epsilon}}{\epsilon})\,dx\,ds
\nonumber\\
&\leq
\mathbb{E}|\psi_0^{m,\epsilon}|_H^2+ M\mathbb{E}\intl_0^t\intl_{\mathbb{S}^1}
|g|^2(\frac{\psi_x^{m,\epsilon}}{\epsilon})\,dx\,ds.\label{eqn:EnergyEstimate_L2}
\end{align}
Moreover, we have following estimate from below
\begin{align}
\mathbb{E}|\psi^{m,\epsilon}|_H^2(t)&-2 \mathbb{E}\intl_0^t(A\psi^{m,\epsilon},\psi^{m,\epsilon})_H\,ds\nonumber\\
&+M\mathbb{E}\intl_0^t\intl_{\mathbb{S}^1}\frac{\psi_x^{m,\epsilon}}{\epsilon} G(\frac{\psi_x^{m,\epsilon}}{\epsilon})\,dx\,ds
\nonumber\\
&\geq
\mathbb{E}|\psi_0^{m,\epsilon}|_H^2.\label{eqn:EnergyEstimate_L2Below}
\end{align}
Furthermore,
\begin{align}
\mathbb{E}|\psi_x^{m,\epsilon}|_H^2(t)&-2 \mathbb{E}\intl_0^t((A\psi^{m,\epsilon})_x,\psi_x^{m,\epsilon})_H\,ds\nonumber\\
&\leq
\mathbb{E}|\psi_{0x}^{m,\epsilon}|_H^2+M_2\mathbb{E}\intl_0^t\intl_{\mathbb{S}^1}
|g|^2(\frac{\psi_x^{m,\epsilon}}{\epsilon})\,dx\,ds,\label{eqn:EnergyEstimate_H1}
\end{align}
where $M_2=\frac{1}{\pi}\suml_{l=1}^{\infty}l^2q_{2l}^2$.
\end{proposition}
\begin{proof}
\begin{itemize}
\item
We apply It\^o formula to deduce that 
\begin{align}\label{eqn:Ito_1}
|\psi^{m,\epsilon}|_H^2(t) &=|\psi^{m,\epsilon}|_H^2(0)+2\intl_0^t(A\psi^{m,\epsilon},\psi^{m,\epsilon})_H\,ds
+\frac{M}{\epsilon^2}\intl_0^t\intl_{\mathbb{S}^1}\psi^{m,\epsilon}\psi_{xx}^{m,\epsilon}|g'|^2(\frac{\psi_x^{m,\epsilon}}{\epsilon})\,dx\,ds\\
&+2\suml_{i=1}^{\infty}\intl_0^t \intl_{\mathbb{S}^1}\psi^{m,\epsilon}g(\frac{\psi_x^{m,\epsilon}}{\epsilon})q_ie_i(y)\,dy d\beta^i(s)+
\suml_{i=1}^{\infty}\intl_0^t \intl_{\mathbb{S}^1}q_i^2|\pi_m\left(g(\frac{\psi_x^{m,\epsilon}}{\epsilon})e_i\right)|^2\,dy\,ds\nonumber
\end{align}
We will need following Lemma:
\begin{lemma}\label{lem:aux_1}
For any $\phi\in C^2(\mathbb{S}^1)$ we have
\begin{equation}
\frac{1}{\epsilon^2}\intl_{\mathbb{S}^1}\phi\phi_{xx}|g'|^2(\frac{\phi_x}{\epsilon})\,dx=-\intl_{\mathbb{S}^1}\frac{\phi_x}{\epsilon} G(\frac{\phi_x}{\epsilon})\,dx,
\end{equation}
where $G$ is defined by formula \eqref{eqn:DefG}.
\end{lemma}
\begin{proof}[Proof of Lemma \ref{lem:aux_1}]
Proof immediately follows by integration by parts. 
\end{proof}
Hence, combining identity \eqref{eqn:Ito_1} and lemma \ref{lem:aux_1} we get
\begin{align}\label{eqn:Ito_2}
|\psi^{m,\epsilon}|_H^2(t)&-2\intl_0^t(A\psi^{m,\epsilon},\psi^{m,\epsilon})_H\,ds+M\intl_0^t\intl_{\mathbb{S}^1}\frac{\psi_x^{m,\epsilon}}{\epsilon} G(\frac{\psi_x^{m,\epsilon}}{\epsilon})\,dx\,ds\\
&=|\psi^{m,\epsilon}|_H^2(0)+2\suml_{i=1}^{\infty}q_i\intl_0^t \intl_{\mathbb{S}^1}\psi^{m,\epsilon}g(\frac{\psi_x^{m,\epsilon}}{\epsilon})e_i(y)\,dy d\beta^i(s)\nonumber\\
&+\suml_{i=1}^{\infty}\intl_0^t \intl_{\mathbb{S}^1}q_i^2|\pi_m\left(g(\frac{\psi_x^{m,\epsilon}}{\epsilon})e_i\right)|^2\,dy\,ds\nonumber
\end{align}
We have that $M_t:=2\suml_{i=1}^{\infty}q_i\intl_0^t \intl_{\mathbb{S}^1}\psi^{m,\epsilon}g(\frac{\psi_x^{m,\epsilon}}{\epsilon})e_i(y)\,dy d\beta^i(s),t\geq 0$ is a local martingale. Define local time $\tau_{m,\epsilon}(k):=\inf\limits_{t\geq 0}\{
|\psi^{m,\epsilon}(t)|_H^2\geq k\}$. Then $N(t):=M(t\wedge \tau_{m,\epsilon}(k)),t\geq 0$ is a martingale.
Let us show that $\tau_{m,\epsilon}(k)\to\infty$ converge to infinity a.s. as $k\to\infty$. We can put in the identity \eqref{eqn:Ito_2} $t:=l\wedge \tau_{m,\epsilon}(k)$ and consider supremum over all $l\leq r$. Then we get 
\begin{align}\label{eqn:Ito_22}
\sup\limits_{l\leq r}|\psi^{m,\epsilon}|_H^2(l\wedge \tau_{m,\epsilon}(k)) &- 2\intl_0^{r\wedge \tau_{m,\epsilon}(k)}(A\psi^{m,\epsilon},\psi^{m,\epsilon})_H\,ds\\
&+M\intl_0^{r\wedge \tau_{m,\epsilon}(k)}\intl_{\mathbb{S}^1}\frac{\psi_x^{m,\epsilon}}{\epsilon} G(\frac{\psi_x^{m,\epsilon}}{\epsilon})\,dx\,ds
\leq|\psi^{m,\epsilon}|_H^2(0)\\
&+2\sup\limits_{l\leq r}\Big|\suml_{i=1}^{\infty}q_i\intl_0^{l\wedge \tau_{m,\epsilon}(k)} \intl_{\mathbb{S}^1}\psi^{m,\epsilon}g(\frac{\psi_x^{m,\epsilon}}{\epsilon})e_i(y)\,dy d\beta^i(s)\Big|\\
&+\suml_{i=1}^{\infty}q_i^2\intl_0^{r\wedge \tau_{m,\epsilon}(k)} \intl_{\mathbb{S}^1}|\pi_m\left(g(\frac{\psi_x^{m,\epsilon}}{\epsilon})e_i\right)|^2\,dy\,ds\nonumber
\end{align}
Consequently, taking expectation of inequality \eqref{eqn:Ito_22}, applying Burkholder-Davis-Gundy inequality and Gronwall inequality we get that $\mathbb{E}\sup\limits_{l\leq r}|\psi^{m,\epsilon}|_H^2(l\wedge \tau_{m,\epsilon}(k))$ is uniformly bounded w.r.t. $k$. Therefore, 
\[\mathbb{P}(\tau_{m,\epsilon}(k)\leq t)=\mathbb{P}(\sup\limits_{l\leq t}|\psi^{m,\epsilon}|_H(l)\geq k)\leq\frac{\mathbb{E}\sup\limits_{l\leq t}|\psi^{m,\epsilon}|_H^2(l\wedge \tau_{m,\epsilon}(k))}{k}\leq\frac{C}{k}\to 0\]
  and a.s. convergence $\tau_{m,\epsilon}(k)\to\infty,k\to\infty$ follows
 (by taking subsequence over $k$). 

Now we have from identity \eqref{eqn:Ito_2}
\begin{align}
\mathbb{E}|\psi^{m,\epsilon}|_H^2(t\wedge \tau_{m,\epsilon}(k))&-
2\mathbb{E}\intl_0^{t\wedge \tau_{m,\epsilon}(k)}(A\psi^{m,\epsilon},\psi^{m,\epsilon})_H\,ds\nonumber\\
&+M\mathbb{E}\intl_0^{t\wedge \tau_{m,\epsilon}(k)}\intl_{\mathbb{S}^1}\frac{\psi_x^{m,\epsilon}}{\epsilon} G(\frac{\psi_x^{m,\epsilon}}{\epsilon})\,dx\,ds \label{eqn:Ito_3}\\
&=\mathbb{E}|\psi^{m,\epsilon}|_H^2(0)+\suml_{i=1}^{\infty}q_i^2\mathbb{E}\intl_0^{t\wedge \tau_{m,\epsilon}(k)} \intl_{\mathbb{S}^1}|\pi_m\left(g\left(\frac{\psi_x^{m,\epsilon}}{\epsilon}\right)e_i\right)|^2\,dy\,ds.\nonumber
\end{align}
Now we take the limit $k\to \infty$ in \eqref{eqn:Ito_3} and notice that projection $|\pi_m|_{\mathcal{L}(H,H)}\leq 1$:
\begin{align}\label{eqn:Ito_4}
\mathbb{E}|\psi^{m,\epsilon}|_H^2(t)&-
2\mathbb{E}\intl_0^{t}(A\psi^{m,\epsilon},\psi^{m,\epsilon})_H^2\,ds
+M\mathbb{E}\intl_0^{t}\intl_{\mathbb{S}^1}\frac{\psi_x^{m,\epsilon}}{\epsilon}G( \frac{\psi_x^{m,\epsilon}}{\epsilon})\,dx\,ds\\
&\leq\mathbb{E}|\psi^{m,\epsilon}|_H^2(0)+ M\mathbb{E}\intl_0^{t} \intl_{\mathbb{S}^1}|g|^2(\frac{\psi_x^{m,\epsilon}}{\epsilon})\,dy\,ds,\nonumber
\end{align}
and the result follows.
\item From identity \eqref{eqn:Ito_3} follows that
\begin{align}
\mathbb{E}|\psi^{m,\epsilon}|_H^2(t\wedge \tau_{m,\epsilon}(k))&-
2\mathbb{E}\intl_0^{t\wedge \tau_{m,\epsilon}(k)}(A\psi^{m,\epsilon},\psi^{m,\epsilon})_H^2\,ds+M\mathbb{E}\intl_0^{t\wedge \tau_{m,\epsilon}(k)}\intl_{\mathbb{S}^1}\frac{\psi_x^{m,\epsilon}}{\epsilon}G(\frac{\psi_x^{m,\epsilon}}{\epsilon})\,dx\,ds\label{eqn:Ito_3_below}\\
&\geq\mathbb{E}|\psi^{m,\epsilon}|_H^2(0).\nonumber
\end{align}
Taking the limit $k\to\infty$ leads to the estimate \eqref{eqn:EnergyEstimate_L2Below}.
\item
We apply It\^o formula to deduce that 
\begin{align}\label{eqn:Ito_5}
|\psi_x^{m,\epsilon}|_H^2(t)&-2\intl_0^t((A\psi^{m,\epsilon})_x,\psi_x^{m,\epsilon})_H\,ds+
\frac{M}{\epsilon^2}\intl_0^t\intl_{\mathbb{S}^1}|g'|^2(\frac{\psi_x^{m,\epsilon}}{\epsilon})|\psi_{xx}^{m,\epsilon}|^2\,dx\,ds\\
&+2\suml_{i=1}^{\infty}q_i\intl_0^t\intl_{\mathbb{S}^1}\psi_{xx}^{m,\epsilon}g(\frac{\psi_{x}^{m,\epsilon}}{\epsilon})e_i\,dx\,d\beta^i(s)\nonumber\\
&=|\psi_x^{m,\epsilon}|_H^2(0)+\suml_{i=1}^{\infty}q_i^2
\intl_0^t\intl_{\mathbb{S}^1}\Big|\pi_m\left[\frac{\psi_{xx}^{m,\epsilon}}{\epsilon}g'(\frac{\psi_x^{m,\epsilon}}{\epsilon})e_i+g(\frac{\psi_x^{m,\epsilon}}{\epsilon})(e_i)_x\right]\Big|^2\,dx\,ds\nonumber
\end{align}

The last term in \eqref{eqn:Ito_5} can be rewritten as follows
\begin{align}\label{eqn:aux_5}
\suml_{i=1}^{\infty}q_i^2\intl_0^t &\intl_{\mathbb{S}^1}\Big|\pi_m\left[\frac{\psi_{xx}^{m,\epsilon}}{\epsilon}g'(\frac{\psi_x^{m,\epsilon}}{\epsilon})e_i\right]\Big|^2\,dx\,ds+\suml_{i=1}^{\infty}q_i^2\intl_0^t\intl_{\mathbb{S}^1}\Big|\pi_m\left[g(\frac{\psi_x^{m,\epsilon}}{\epsilon})(e_i)_x\right]\Big|^2\,dx\,ds\\
&+2\suml_{i=1}^{\infty}q_i^2\intl_0^t\intl_{\mathbb{S}^1}\pi_m\left[\frac{\psi_{xx}^{m,\epsilon}}{\epsilon}g'(\frac{\psi_x^{m,\epsilon}}{\epsilon})e_i\right]\pi_m\left[g(\frac{\psi_x^{m,\epsilon}}{\epsilon})(e_i)_x\right]\,dx\,ds\nonumber
\end{align}
\begin{align*}
&=\suml_{i=1}^{\infty}q_i^2\intl_0^t\intl_{\mathbb{S}^1}\Big|\pi_m\left[\frac{\psi_{xx}^{m,\epsilon}}{\epsilon}g'(\frac{\psi_x^{m,\epsilon}}{\epsilon})e_i\right]\Big|^2\,dx\,ds+\suml_{i=1}^{\infty}q_i^2\intl_0^t\intl_{\mathbb{S}^1}\Big|\pi_m\left[g(\frac{\psi_x^{m,\epsilon}}{\epsilon})(e_i)_x\right]\Big|^2\,dx\,ds\\
&+2\suml_{i=1}^{\infty}q_i^2\intl_0^t\intl_{\mathbb{S}^1}(\id-\pi_m)\left[\frac{\psi_{xx}^{m,\epsilon}}{\epsilon}g'(\frac{\psi_x^{m,\epsilon}}{\epsilon})e_i\right](\id-\pi_m)\left[g(\frac{\psi_x^{m,\epsilon}}{\epsilon})(e_i)_x\right]\,dx\,ds
\end{align*}
\begin{align*}
&\leq\suml_{i=1}^{\infty}q_i^2\intl_0^t\intl_{\mathbb{S}^1}\Big|\pi_m\left[\frac{\psi_{xx}^{m,\epsilon}}{\epsilon}g'(\frac{\psi_x^{m,\epsilon}}{\epsilon})e_i\right]\Big|^2\,dx\,ds+\suml_{i=1}^{\infty}q_i^2\intl_0^t\intl_{\mathbb{S}^1}\Big|\pi_m\left[g(\frac{\psi_x^{m,\epsilon}}{\epsilon})(e_i)_x\right]\Big|^2\,dx\,ds\\
&+\suml_{i=1}^{\infty}q_i^2\intl_0^t\intl_{\mathbb{S}^1}\Big|((\id-\pi_m)\left[\frac{\psi_{xx}^{m,\epsilon}}{\epsilon}g'(\frac{\psi_x^{m,\epsilon}}{\epsilon})e_i\right]\Big|^2\,dx\,ds+\suml_{i=1}^{\infty}q_i^2\intl_0^t\intl_{\mathbb{S}^1}\Big|(\id-\pi_m)\left[g(\frac{\psi_x^{m,\epsilon}}{\epsilon})(e_i)_x\right]\Big|^2\,dx\,ds\\
&=\suml_{i=1}^{\infty}q_i^2\intl_0^t\intl_{\mathbb{S}^1}\frac{(\psi_{xx}^{m,\epsilon})^2}{\epsilon^2} |g'|^2(\frac{\psi_x^{m,\epsilon}}{\epsilon})e_i^2\,dx\,ds+\suml_{i=1}^{\infty}q_i^2\intl_0^t\intl_{\mathbb{S}^1}|g|^2(\frac{\psi_x^{m,\epsilon}}{\epsilon})(e_i)_x^2\,dx\,ds
\end{align*}
where first identity follows from the fact that 
\begin{equation*}
2\suml_{i=1}^{\infty}q_i^2\intl_0^t\intl_{\mathbb{S}^1}\psi_{xx}^{m,\epsilon}g'(\frac{\psi_x^{m,\epsilon}}{\epsilon})e_i g(\frac{\psi_x^{m,\epsilon}}{\epsilon})(e_i)_x\,dx\,ds=0,
\end{equation*}
(because $\suml_{i=1}^{\infty}q_i^2e_i(e_i)_x=\frac12\left(\suml_{i=1}^{\infty}q_i^2|e_i|^2\right)_x=0$)
and second inequality is a consequence of Cauchy-Schwartz inequality.

Combining formula \eqref{eqn:Ito_5} with inequality \eqref{eqn:aux_5} we can deduce that
\begin{align}\label{eqn:Ito_6}
|\psi_x^{m,\epsilon}|_H^2(t)&-2\intl_0^t((A\psi^{m,\epsilon})_x,\psi_x^{m\epsilon})_H^2\,ds\\
&+2\suml_{i=1}^{\infty}q^i\intl_0^t\intl_{\mathbb{S}^1}\psi_{xx}^{m,\epsilon}g(\frac{\psi_{x}^{m,\epsilon}}{\epsilon})e_i\,dx\,d\beta^i(s)\nonumber\\
&\leq|\psi_x^{m,\epsilon}|_H^2(0)+M_2\intl_0^t\intl_{\mathbb{S}^1}|g|^2(\frac{\psi_{x}^{m,\epsilon}}{\epsilon})\,dx\,ds,\nonumber
\end{align}
where $M_2=\frac{1}{\pi}\suml_{l=1}^{\infty}l^2q_{2l}^2$.
Conclusion of the proof follows in the same fashion as in part 1 (i.e. considering appropriate local time to stop local martingale in formula \eqref{eqn:Ito_6}, taking expectation and the limit).
\end{itemize}
\end{proof}
\begin{corollary}
Assume that 
there exists constant $C>0$ such that $g\in L^{\infty}(\mathbb{R})$, $g'\in L^2\cap L^{\infty}(\mathbb{R})$ satisfies
\begin{equation}\label{eqn:TailCondition}
\left(\int\limits_{-\infty}^{-z}+\int\limits_{z}^{\infty}\right)|g'|^2(y)\,dy\leq \frac{C}{z},z>0,
\end{equation}
and 
\begin{equation}\label{eqn:NontrivialityCondition}
\kappa=\min\{\int\limits_{0}^{\infty}|g'|^2(y)\,dy,\int\limits_{-\infty}^{0}|g'|^2(y)\,dy\}>0.
\end{equation} 
Then there exists $C=C(t,\alpha,\beta,|g|_{L^{\infty}},\psi_0)>0$ independent upon $m$ and $\epsilon$  such that
\begin{equation}\label{eqn:UniformEnergyEst}
\intl_0^t\mathbb{E}|\psi_x^{m,\epsilon}|_{L^1}\,ds\leq C\frac{\epsilon}{\kappa}.
\end{equation}
\end{corollary}
\begin{proof}
By boundedness of $g$, dissipativity of $A$ \eqref{eqn:DissipatCondition} and a priori estimates \eqref{eqn:EnergyEstimate_L2}, \eqref{eqn:EnergyEstimate_H1} we have that
\begin{equation}
\mathbb{E}\intl_0^t\intl_{\mathbb{S}^1}\frac{\psi_x^{m,\epsilon}}{\epsilon} G(\frac{\psi_x^{m,\epsilon}}{\epsilon})\,dx\,ds\leq C(t,\alpha, \beta  , |g|_{L^{\infty}}).
\end{equation}
Hence we have that
\begin{equation}\label{eqn:AprioriMain_1}
\mathbb{E}\intl_0^t\intl_{\{\psi_x^{m,\epsilon}\geq 0\}}\frac{\psi_x^{m,\epsilon}}{\epsilon} G(\frac{\psi_x^{m,\epsilon}}{\epsilon})\,dx\,ds+
\mathbb{E}\intl_0^t\intl_{\{\psi_x^{m,\epsilon}<0\}}\frac{\psi_x^{m,\epsilon}}{\epsilon} G(\frac{\psi_x^{m,\epsilon}}{\epsilon})\,dx\,ds\leq C(t,\alpha, \beta , |g|_{L^{\infty}}).
\end{equation}
Consequently, condition \eqref{eqn:TailCondition} together with the estimate \eqref{eqn:AprioriMain_1} gives us that
\begin{align*}
\mathbb{E}\intl_0^t\intl_{\{\psi_x^{m,\epsilon}\geq 0\}}&\frac{\psi_x^{m,\epsilon}}{\epsilon}\int\limits_{0}^{\infty}|g'|^2(y)\,dy \,dx\,ds\\
&+
\mathbb{E}\intl_0^t\intl_{\{\psi_x^{m,\epsilon}< 0\}}-\frac{\psi_x^{m,\epsilon}}{\epsilon}\int\limits_{-\infty}^{0}|g'|^2(y)\,dy \,dx\,ds\leq C(t,\alpha, \beta , |g|_{L^{\infty}}),
\end{align*}
and the result follows. 
\end{proof}
The a priori estimates of Proposition \ref{prop:Galerkin_Estimate} are uniform w.r.t. both parameter $\epsilon$ and dimension $m$ of the approximation space $H_m$. The next a priori estimate will give us bound on fractional time derivative of the solution. The estimate is not uniform w.r.t. $\epsilon$.
\begin{lemma}\label{lem:FracDerivative}
For any $\epsilon>0$, $T>0$, $\alpha\in\left(0,\frac{1}{2}\right)$ there exists $C(\epsilon,T,\alpha)$ such that 
\begin{equation}
\mathbb{E}|\psi^{m,\epsilon}|_{H^{\alpha,2}(0,T;L^2(\mathbb{S}^1))}^2\leq C(\epsilon,T,\alpha).\label{eqn:TimeFracDer_Est}
\end{equation}
\end{lemma}
\begin{proof}
By definition \eqref{eqn:GalerkinApprox} of Galerkin approximation $\psi^{m,\epsilon}$ has representation
\[
\psi^{m,\epsilon}(t)=\psi^{m,\epsilon}(0)+\intl_0^t\pi_m(A^{\epsilon}(\pi_m\psi^{m,\epsilon}))\,ds+\intl_0^t\pi_m(\sigma^{\epsilon}(\pi_m\psi^{m,\epsilon})dW_s^Q).
\]
Now for any fixed $\epsilon>0$ the drift term is bounded in $L^2(\Omega,H^{1,2}(0,T;L^2(\mathbb{S}^1)))$ by a priori estimate \eqref{eqn:EnergyEstimate_H1}. Furthermore, diffusion term is bounded in 
$L^2(\Omega,H^{\alpha,2}(0,T;L^2(\mathbb{S}^1)))$ for any $\alpha\in\left(0,\frac{1}{2}\right)$ by Lemma 2.1, of \cite{FlandoliGatarek95}. 
\end{proof}
Now we are ready to converge $m$ to infinity in Galerkin approximation \eqref{eqn:GalerkinApprox} and show existence of strong solution of equation \eqref{Mainmodel} for any $\epsilon>0$.

\section{Main result}

\begin{proposition}\label{prop:ExistEpsPositive}
Assume that $g\in L^{\infty}(\mathbb{R})$, $g'\in L^2\cap L^{\infty}(\mathbb{R})$ satisfies conditions \eqref{eqn:TailCondition} and \eqref{eqn:NontrivialityCondition}.
Then there exist global strong solution $\psi^{\epsilon}$ of the system \eqref{Mainmodel_Itoform} and $C(t,\alpha,\beta,M,|g|_{L^{\infty}},\psi_0)>0$ such that 
\begin{equation}\label{eqn:UniformEnergyEst_2}
\intl_0^t\mathbb{E}|\psi_x^{\epsilon}|_{L^1}\,ds\leq 
C\frac{\epsilon}{\kappa},
\end{equation}
where $\kappa$ is defined by \eqref{eqn:NontrivialityCondition}.
In particular, we have that
\[
\limsup_{\epsilon\to 0}\intl_0^t\mathbb{E}|\psi_x^{\epsilon}|_{L^1}\,ds=0.
\]
\end{proposition}
\begin{theorem}\label{thm:Mainthm_1}
Assume that $A^*(1)=0$ i.e. 
\begin{equation}\label{eqn:Nodrift}
\intl_{\mathbb{S}^1}A\phi\,dx=0,\,\,\forall\phi\in C^{\infty}(\mathbb{S}^1).
\end{equation} 
Then there exists a martingale $\psi\in L^2(\Omega, C([0,T],\mathbb{R}))$ such that for any $\phi\in C^{\infty}([0,T]\times\mathbb{S}^1)$ we have
\[\lim_{\epsilon\to 0}
\mathbb{E}\left|\int\limits_0^t\int\limits_{\mathbb{S}^1}(\psi^{\epsilon}(s,x,\cdot)-\psi(s,\cdot))\phi(s,x)\,dx\,ds\right|=0.
\]
Furthermore,
\[
\mathbb{E}\psi(t)=\frac{1}{2\pi}\int\limits_{\mathbb{S}^1}\psi_0(x)\,dx.
\]
\end{theorem}
\begin{remark}
It remains an open problem to find quadratic variation of $\psi$. 
\end{remark}
\begin{remark}
Assumption \eqref{eqn:Nodrift} in Theorem \ref{thm:Mainthm_1} is for simplicity. Otherwise, we would get in the limit $\epsilon\to 0$ martingale with additional drift term. The structure of the drift would depend on the exact form of the operator $A$.
\end{remark}

\section{Proofs of Proposition \ref{prop:ExistEpsPositive} and Theorem \ref{thm:Mainthm_1}}

\begin{proof}[Proof of Proposition \ref{prop:ExistEpsPositive}]
Let $\{\psi^{m,\epsilon}\}_{m\in\mathbb N,\epsilon>0}$ be Galerkin approximation introduced in \eqref{eqn:GalerkinApprox}. According to Proposition \ref{prop:Galerkin_Estimate} and Lemma \ref{lem:FracDerivative} we have following a priori estimate 
\[
\sup_{m\in\mathbb{N}}\left[\mathbb{E}|\psi^{m,\epsilon}|_{H^{\alpha,2}(0,T;L^2(\mathbb{S}^1))}^2+\mathbb{E}|\psi^{m,\epsilon}|_{L^2(0,T;H^{2,2}(\mathbb{S}^1))\cap C(0,T;H^{1,2}(\mathbb{S}^1))}^2\right]<\infty,\quad\alpha\in \left(0,\frac12\right)
\]
Space $L^2(0,T;H^{2,2}(\mathbb{S}^1))\cap H^{\alpha,2}(0,T;L^2(\mathbb{S}^1))$,  $0<\alpha<\frac12$ is compactly embedded in \\
$L^2(0,T;H^{1,2}(\mathbb{S}^1))$ by Theorem 2.1 from \cite{FlandoliGatarek95}. 
Consequently, family of probability laws
$\mathcal{L}(\psi^{m,\epsilon})$ is tight in $L^2(0,T;H^{1,2}(\mathbb{S}^1))$. Hence, there exists subsequence $\psi^{m,\epsilon}$ (denoted by the same letter) such that  $\mathcal{L}(\psi^{m,\epsilon})$ weakly converges in $L^2(0,T;H^{1,2}(\mathbb{S}^1))$ (for fixed $\epsilon>0$).

By the Skorokhod embedding theorem  (cf. \cite{IkWat81}, p.9) there exists stochastic basis $(\Omega, \mathcal{F},\{\mathcal{F}_t\}_{t\geq 0},\mathbb{P})$ and random variables $\widetilde{\psi}^{\epsilon}$, $\widetilde{\psi}^{m,\epsilon}$, $m\in\mathbb{N}$ such that $\widetilde{\psi}^{m,\epsilon}\to\widetilde{\psi}^{\epsilon}$ in $L^2(0,T;H^{1,2}(\mathbb{S}^1))$ $\mathbb{P}$-a.s. and we have that the probability laws of $\widetilde{\psi}^{m,\epsilon}$ and $\psi^{m,\epsilon}$ on $L^2(0,T;H^{1,2}(\mathbb{S}^1))$ are the same. Therefore, $\widetilde{\psi}^{m,\epsilon}$ satisfy the same a priori estimate as $\psi^{m,\epsilon}$. Consequently,
\begin{equation}
\widetilde{\psi}^{\epsilon}\in L^2(0,T;H^{2,2}(\mathbb{S}^1))\cap C(0,T;H^{1,2}(\mathbb{S}^1))\,\,\mathbb{P}\mbox{-a.s.},\label{regularity} 
\end{equation}
and $\widetilde{\psi}^{m,\epsilon}\to\widetilde{\psi}^{\epsilon}$ in $L^2(\Omega\times[0,T],H^{2,2}(\mathbb{S}^1))$ weakly. Define 
\[
M^{m,\epsilon}(t):=\widetilde{\psi}^{m,\epsilon}(t)-\pi_m\widetilde{\psi}^{m,\epsilon}(0)
-\intl_0^t\pi_m(A^{\epsilon}(\pi_m\widetilde{\psi}^{m,\epsilon}))\,dt, t\geq 0.
\]
Then $\{M^{m,\epsilon}\}_{t\geq 0}$ is a a square integrable martingale with respect to the filtration $(\mathcal{G}^{m,\epsilon})_t=\sigma(\{\widetilde{\psi}^{m,\epsilon}(s),s\leq t\})$ with quadratic variation 
\[
<<M^{m,\epsilon}>>(t)=\sum\limits_{i=1}^{\infty}q_i^2\intl_0^t|\pi_m(\sigma(\widetilde{\psi}^{m,\epsilon})e_i)|^2\,ds.
\]
Indeed, since the laws $\mathcal{L}(\widetilde{\psi}^{m,\epsilon})$ and 
$\mathcal{L}(\psi^{m,\epsilon})$ are the same we have that
for all $0\leq s\leq t$, $\la \in C_b(L^2([0,T),H^{2,2}(\mathbb{S}^1))),\phi,\gamma\in C^{\infty}(\mathbb{S}^1)$
\begin{equation}\label{martprop_1}
\mathbb{E}[(M^{m,\epsilon}(t)-M^{m,\epsilon}(s),\phi)\la(\widetilde{\psi}^{m,\epsilon}|_{[0,s]})]=0,
\end{equation}
and
\begin{equation}\label{martprop_2}
\begin{aligned}
\mathbb{E}\la(\widetilde{\psi}^{m,\epsilon}|_{[0,s]})[(M^{m,\epsilon}(t),\phi)(M^{m,\epsilon}(t),\gamma)&-(M^{m,\epsilon}(s),\phi)(M^{m,\epsilon}(s),\gamma)\\
&-\sum\limits_{i=1}^{\infty}q_i^2\intl_s^t(\pi_m(\sigma(\widetilde{\psi}^{m,\epsilon})e_i)\phi,\pi_m(\sigma(\widetilde{\psi}^{m,\epsilon})e_i)\gamma)\,ds]=0.
\end{aligned}
\end{equation}
It remains to take the limit $m\to\infty$ in equalities \eqref{martprop_1} and \eqref{martprop_2}.  By a priori estimates \eqref{eqn:EnergyEstimate_H1},\eqref{eqn:EnergyEstimate_L2}, all terms in equalities \eqref{martprop_1} and \eqref{martprop_2} are uniformly integrable w.r.t. $\omega$. Thus we need to show convergence $\mathbb{P}$-a.s.. Notice that for any test function $\phi \in C^{\infty}(\mathbb{S}^1)$ the drift term $\left(\intl_0^t\pi_m(A^{\epsilon}(\widetilde{\psi}^{m,\epsilon}))\,ds,\phi\right)_{L^2}$ can be rewritten as follows
\begin{equation}\label{eqn:aux_9}
\left(\intl_0^t\pi_m(A^{\epsilon}(\widetilde{\psi}^{m,\epsilon}))\,ds,\phi\right)_{L^2}=-\frac{M}{2}\intl_0^t(\frac{\pi_m\phi_x}{\epsilon},G(\frac{\widetilde{\psi}^{m,\epsilon}}{\epsilon}))_{L^2}\,ds+\intl_0^t\left(\widetilde{\psi}^{m,\epsilon},A^*\phi\right)_{L^2}\,ds,
\end{equation}
where $G$ is given by \eqref{eqn:DefG}. Indeed, representation \eqref{eqn:aux_9} follows from  integration by parts. Consequently, convergence of the RHS term in \eqref{eqn:aux_9} follows from global Lipshitz property of function $G$. Similarly, we can show convergence of quadratic variation. 
Now the existence of weak solution follows from representation Theorem for martingales (Theorem 8.2, p. 220 \cite{DZ1992} ). The weak solution is a strong one by the regularity property \eqref{regularity} and integration by parts formula. The identity \eqref{eqn:UniformEnergyEst_2} follows from identity \eqref{eqn:UniformEnergyEst}.
\end{proof}
\begin{proof}[Proof of Theorem \ref{thm:Mainthm_1}]
We can represent $\psi^{\epsilon}$ as follows
\[
\psi^{\epsilon}=\left(\psi^{\epsilon}-\frac{1}{2\pi}\int\limits_{\mathbb{S}^1}\psi^{\epsilon}\,dx\right)+\frac{1}{2\pi}\int\limits_{\mathbb{S}^1}\psi^{\epsilon}\,dx.
\]
Let 
\[
\chi(x):=\intl_0^x\phi(y)\,dy-\frac{x}{2\pi}\intl_0^{2\pi}\phi(y)\,dy, x\in [0,2\pi).
\] 
Note that $\chi_x=\phi-\frac{1}{2\pi}\intl_0^{2\pi}\phi(y)\,dy$.
Consequently, we have by integration by parts that
\begin{align*}
\Big| &\int\limits_0^T\int\limits_{\mathbb{S}^1}\big(\psi^{\epsilon}-\frac{1}{2\pi}\int\limits_{\mathbb{S}^1}\psi^{\epsilon}\,dx\big)\phi\,dx\,ds\Big| \\
&=\left|\int\limits_0^T\int\limits_{\mathbb{S}^1}\left(\psi^{\epsilon}-\frac{1}{2\pi}\int\limits_{\mathbb{S}^1}\psi^{\epsilon}\,dx\right)\chi_x\,dx\,ds\right|\\
&=\left|\int\limits_0^T\int\limits_{\mathbb{S}^1}\psi_x^{\epsilon}\chi\,dx\,ds\right|\leq ||\chi||_{L^{\infty}([0,T]\times \mathbb{S}^1)}\int\limits_0^T\int\limits_{\mathbb{S}^1}|\psi_x^{\epsilon}|\,dx\,ds
\end{align*}
which converges to $0$ by Proposition \ref{prop:ExistEpsPositive}.
Hence it remains to find the limit of $\epsilon$ converging to zero of $M^{\epsilon}(t):=\frac{1}{2\pi}\int\limits_{\mathbb{S}^1}\psi^{\epsilon}\,dx,t\geq 0$.
First let us notice that we have the following representation of $M^{\epsilon}$:  
\begin{equation}\label{eqn:MartRepr}
M^{\epsilon}(t)=\frac{1}{2\pi}\int\limits_{\mathbb{S}^1}\psi_0(x)\,dx+\frac{1}{2\pi}\int\limits_0^t\int\limits_{\mathbb{S}^1}g\left(\frac{\psi^{\epsilon}_x}{\epsilon}\right) \,dx\,dW_s^Q,\quad t\geq 0,
\end{equation}
where we have used assumption \eqref{eqn:Nodrift} to cancel the drift part.
Thus we get that $M_{\epsilon}$ is a sequence of square integrable martingales and by the Burkholder-Davies-Gundy inequality 
\[
\sup\limits_{\epsilon>0}\mathbb{E}\sup\limits_{t\in [0,T]}|M^{\epsilon}(s)|^p<\infty,\quad p\geq 1.
\]
Furthermore, we can deduce from representation \eqref{eqn:MartRepr} that
\[
\sup\limits_{\epsilon>0}\mathbb{E}|M^{\epsilon}|_{W^{\alpha,p}([0,T],\mathbb{R})}^p<\infty, \quad \alpha\in (0,\frac{1}{2})\,,p>1.
\]
Hence, by compact embedding theorem we have that the Martingale sequence $M^{\epsilon}$ is tight in $C([0,T],\mathbb{R})$. Consequently, by the Prokhorov Theorem it converges in law to the process $\psi$ in $C([0,T],\mathbb{R})$. In particular,
\[
\mathbb{E}\left|\int\limits_0^T M^{\epsilon}(s)\int\limits_{\mathbb{S}^1}\phi(s,x)\,dx\,ds-\int\limits_0^T \psi(s)\int\limits_{\mathbb{S}^1}\phi(s,x)\,dx\,ds\right|\to 0,\quad\epsilon\to 0.
\]
The process $\psi$ is a square integrable martingale (See, for instance proposition 1.12, Chapter 9 of \cite{JacodShiryaev2003}) with expectation $\mathbb{E}\psi(t)=\frac{1}{2\pi}\int\limits_{\mathbb{S}^1}\psi_0(x)\,dx$.
\end{proof}

\section{Examples and counterexamples}\label{sec:counterexample}

Here we will consider several examples of $g$ which does not satisfy conditions of the Theorem \ref{thm:Mainthm_1}. 
\begin{enumerate}
\item $g(z)=|z|^{\gamma}, \gamma>1$. 
\begin{proposition}\label{prop:Mainthm_2}
Assume that $g(z)=|z|^{\gamma}, \gamma>1$, $A$ satisfies assumption \eqref{eqn:DissipatCondition}. 
Then there exist global strong solution $\psi^{\epsilon}$ of the system \eqref{Mainmodel_Itoform} and $C=C(t,\alpha,\beta,\gamma,n,|g|_{L^{\infty}},\psi_0)>0$ such that 
\begin{equation}\label{eqn:UniformEnergyEst_2a}
\intl_0^t\mathbb{E}|\psi_x^{\epsilon}|_{L^{2\gamma}}^{2\gamma}\,ds\leq 
C \epsilon^{2\gamma}.
\end{equation}
In particular, we have that
\[
\limsup_{\epsilon\to 0}\intl_0^t\mathbb{E}|\psi_x^{\epsilon}|_{L^{2\gamma}}^{2\gamma}\,ds=0.
\]
\end{proposition}
\begin{proof}
The proof is quite similar to the proof of proposition \ref{prop:ExistEpsPositive}. So we will explain here only the differences.  Define operator 
$A^{\epsilon}: \mathbb{H}^{2,4}(\mathbb{S}^1)\cap \mathbb{H}^{1,4\gamma-4}\to H,\quad  A^{\epsilon}(f):=A(f)+\dfrac{M\gamma^2f_x^{2\gamma-2}}{2\epsilon^{2\gamma}}f_{xx}$ and we can conclude existence of local solution $\psi^{m,\epsilon}$ as before. Then It\^o formula allow us to deduce that
\begin{align}
\mathbb{E}|\psi^{m,\epsilon}(t)|_{H}^2&-2\mathbb{E}\intl_0^t(A\psi^{m,\epsilon},\psi^{m,\epsilon})_{H}\,ds+M\frac{(\gamma -1)^2}{2\gamma -1}\mathbb{E}\intl_0^t\intl_{\mathbb{S}^1}\frac{(\psi_x^{m,\epsilon})^{2\gamma}}{\epsilon^{2\gamma}}\,dx\,ds\nonumber\\
&\leq\mathbb{E}|\psi_0^{m,\epsilon}|_{H}^2,m\in \mathbb{N}.\label{eqn:EnergyEstimate_L2_ca}
\end{align}
Consequently, we can deduce a priori estimate for higher order norm i.e. we get 
\begin{align}
\mathbb{E}|\psi_x^{m,\epsilon}|_H^2(t)&-2 \mathbb{E}\intl_0^t((A\psi^{m,\epsilon})_x,\psi_x^{m,\epsilon})_H\,ds\nonumber\\
&\leq
\mathbb{E}|\psi_{0x}^{m,\epsilon}|_H^2+\frac{(2\gamma -1)M_2}{M(\gamma -1)^2}\mathbb{E}|\psi_0^{m,\epsilon}|_{H}^2,m\in \mathbb{N}.\label{eqn:EnergyEstimate_H1_ca}
\end{align}
Now we have uniform in $m,\epsilon$ estimates for $\psi^{m,\epsilon}$ and homogenisation inequality (in \eqref{eqn:EnergyEstimate_L2_ca}).  
\end{proof}
\item $g(z)=z$. In this case homogenisation doesn't hold  as following elementary example shows. Let $A=\partial^2_{xx}$ with periodic boundary conditions and assume that noise $W^Q=\beta$ is a one dimensional Wiener process. Then system 
\eqref{Mainmodel_Itoform} has a unique solution of the form $\psi^{\epsilon}(t,x)=\psi_0(x+\frac{\beta(t)}{\epsilon}),t\geq 0, x\in\mathbb{S}^1$. Consequently, integral $\intl_0^t\mathbb{E}|\psi_x^{\epsilon}|_{L^1}\,ds$ does not depend upon $\epsilon$.
\item $g(z)=\sin{z}$.
\begin{proposition}\label{prop:Mainthm_3}
Assume that $g(z)=\sin{z}$, $A$ satisfies assumption \eqref{eqn:DissipatCondition}. 
Then there exist global strong solution $\psi^{\epsilon}$ of the system \eqref{Mainmodel_Itoform} and $C=C(t,\alpha,\beta,\gamma,M,|g|_{L^{\infty}},\psi_0)>0$ such that 
\begin{equation}\label{eqn:UniformEnergyEst_2b}
\intl_0^t\mathbb{E}|\psi_x^{\epsilon}|_{L^{2}}^{2}\,ds\leq 
C \epsilon^{2}.
\end{equation}
In particular, we have that
\[
\limsup_{\epsilon\to 0}\intl_0^t\mathbb{E}|\psi_x^{\epsilon}|_{L^{2}}^{2}\,ds=0.
\]
\end{proposition}
\begin{proof}
The result follows from It\^o formula and energy estimate.
\end{proof}
\end{enumerate}
\section{Appendix}\label{sec:ItoCorrectionTerm}
In the appendix we formally calculate It\^o correction term for equation \eqref{Mainmodel}. From \eqref{Mainmodel} we have 
\begin{align}
\frac{1}{2}<g(\frac{\psi_x^{\epsilon}}{\epsilon}), W^Q>_t&=\frac{1}{2}<\intl_{0}^{\cdot}\frac{1}{\epsilon}g'(\frac{\psi_x^{\epsilon}}{\epsilon})d\psi_x^{\epsilon},W^Q>_t\nonumber\\
&=\frac{1}{2}<\intl_{0}^{\cdot}\frac{1}{\epsilon}g'(\frac{\psi_x^{\epsilon}}{\epsilon})(g'(\frac{\psi_x^{\epsilon}}{\epsilon})\frac{\psi_{xx}^{\epsilon}}{\epsilon}dW^Q+g(\frac{\psi_x^{\epsilon}}{\epsilon})dW_x^Q), W^Q>_t\nonumber\\
&=\intl_{0}^{t}\frac{1}{2\epsilon^2}|g'|^2(\frac{\psi_x^{\epsilon}}{\epsilon})\psi_{xx}^{\epsilon}\rho^{Q}(x)ds+\intl_{0}^{t}\frac{1}{4\epsilon}g'g(\frac{\psi_x^{\epsilon}}{\epsilon})\rho_x^{Q}(x)ds,t\geq 0\label{eqn:ItoCorr_1}
\end{align}
where
\[
\rho^Q=\sum\limits_{n=1}^{\infty}q_k^2e_k^2.
\]
Note that we can rewrite $\rho^Q$ as follows
\[
\rho^Q(x)=\frac{1}{2\pi}+\frac{1}{\pi}\suml_{n=1}^{\infty}q_{2n+1}^2\cos^2{nx}+q_{2n}^2\sin^2{nx}=\frac{1}{2\pi}\sum\limits_{n=1}^{\infty}q_n^2
+\sum\limits_{n=1}^{\infty}(q_{2n+1}^2-q_{2n}^2)\cos{2nx}, x\in[0,2\pi).
\]
Consequently, condition \eqref{eqn:QCondition} implies that  
\begin{equation}
\rho^Q=\frac{1}{2\pi}\sum\limits_{n=1}^{\infty}q_n^2:=M,\, \rho_x^{Q}=0.\label{eqn:ItoCorr_2}
\end{equation}
Combining equalities \eqref{eqn:ItoCorr_1} and \eqref{eqn:ItoCorr_2} we get 
\begin{equation}
\frac{1}{2}<g(\frac{\psi_x^{\epsilon}}{\epsilon}), W^Q>_t=\intl_{0}^{t}\frac{M|g'|^2(\frac{\psi_x^{\epsilon}}{\epsilon})}{2\epsilon^2}\psi_{xx}^{\epsilon}ds
\end{equation}

{\bf Acknowledgment:} We thank Z. Brz\'{e}zniak and M. R\"{o}ckner for useful discussions and attention to the work. This work was supported by the ARC Discovery grant DP120101886.

\end{document}